\documentclass{amsart}

\usepackage[alphabetic,initials,nobysame]{amsrefs}
\usepackage{amssymb,mathtools,mathrsfs}
\usepackage{hyperref} 
\usepackage[shortlabels]{enumitem}
\usepackage{caption}
\usepackage{tikz}

\BibSpec{collection.article}{%
	+{}  {\PrintAuthors}				{author}
	+{, } {}                     		{title}
	+{.} { }                            {part}
	+{:} { \textit}                     {subtitle}
	+{,} { \PrintContributions}         {contribution}
	+{,} { \PrintConference}            {conference}
	+{}  {\PrintBook}                   {book}
	+{. In }{\textit}                   {booktitle}
	+{,} { }							{series}
	+{,} { pp.~}                        {pages}
	+{,} { }							{publisher}
	+{,} { }							{address}
	+{,} { \PrintDateB}                 {date}
	+{,} { }                            {status}
	+{,} { \PrintDOI}                   {doi}
	+{,} { available at \eprint}        {eprint}
	+{}  { \parenthesize}               {language}
	+{}  { \PrintTranslation}           {translation}
	+{;} { \PrintReprint}               {reprint}
	+{.} { }                            {note}
	+{.} {}                             {transition}
	+{}  {\SentenceSpace \PrintReviews} {review}
}

\theoremstyle{plain}
\newtheorem{theorem}{Theorem}
\newtheorem{lemma}[theorem]{Lemma}
\newtheorem{conjecture}[theorem]{Conjecture}
\newtheorem{prop}[theorem]{Proposition}

\theoremstyle{definition}
\newtheorem{definition}[theorem]{Definition}

\theoremstyle{remark}

\newtheorem{question}[theorem]{Question}

\numberwithin{equation}{section}
\numberwithin{theorem}{section}

\newcommand{\br}{\overline}

\newcommand{\C}{\mathbb C}

\newcommand{\N}{\mathbb N}

\DeclareMathOperator{\dist}{{\mathrm{dist}}}
\DeclareMathOperator{\diam}{{\mathrm{diam}}}

\DeclareMathOperator{\inter}{{\mathrm{int}}}

\DeclareMathOperator{\md}{\mathrm{Mod}}

\DeclareMathOperator{\area}{\mathrm{Area}}

\DeclareMathOperator{\NED}{\mathit{NED}}
\DeclareMathOperator{\CNED}{\mathit{CNED}}

\begin{document}
\title[Rigidity and continuous extension in circle domains]{Rigidity and continuous extension for conformal maps of circle domains}
\author{Dimitrios Ntalampekos}
\address{Mathematics Department, Stony Brook University, Stony Brook, NY 11794, USA.}
\thanks{The author is partially supported by NSF Grant DMS-2000096.}
\email{dimitrios.ntalampekos@stonybrook.edu}
\date{\today}

\begin{abstract}
We present sufficient conditions so that a conformal map between planar domains whose boundary components are Jordan curves or points has a continuous or homeomorphic extension to the closures of the domains. Our conditions involve the notions of cofat domains and $\CNED$ sets, i.e., countably negligible for extremal distance, recently introduced by the author. We use this result towards establishing conformal rigidity of a class of circle domains. A circle domain is conformally rigid if every conformal map onto another circle domain is the restriction of a M\"obius transformation.  We show that circle domains whose point boundary components are $\CNED$ are conformally rigid. This result is the strongest among all earlier works and provides substantial evidence towards the rigidity conjecture of He--Schramm, relating the problems of conformal rigidity and removability. 
\end{abstract}
\maketitle

\section{Introduction}

One of the most intriguing open problems in complex analysis is \textit{Koebe's conjecture}, predicting that every domain in the Riemann sphere is conformally equivalent to a \textit{circle domain}, a domain whose boundary components are circles or points. The best known result so far is the validity of the conjecture for countably connected domains, established in a seminal work of He--Schramm \cite{HeSchramm:Uniformization}. See the references in that paper for the history of the problem and also \cites{Schramm:transboundary,Rajala:koebe} for other approaches to the same result.

Although the existence in Koebe's conjecture is open in the general case, it is well known that uniqueness fails. Whenever the uniqueness fails for a domain $U$, there exists a conformal map $f_1$ from $U$ onto a circle domain $V_1$ and a conformal map $f_2$ from $U$ onto a circle domain $V_2$ such that the composition $f_2\circ f_1^{-1}$ is not a M\"obius transformation. A circle domain is \textit{conformally rigid} if every conformal map onto another circle domain is the restriction of a M\"obius transformation of the sphere. Thus, $V_1$ and $V_2$ are not rigid. 

The class of rigid circle domains includes finitely connected domains \cite{Koebe:FiniteUniformization}, countably connected domains \cite{HeSchramm:Uniformization}, domains whose boundary has $\sigma$-finite Hausdorff $1$-measure \cite{HeSchramm:Rigidity}, and domains whose quasihyperbolic distance satisfies a certain integrability condition \cite{NtalampekosYounsi:rigidity}, which is valid, for example, in John and H\"older domains. On the other hand, circle domains whose boundary has positive area are not rigid, as follows from work of Sibner \cite{Sibner:KoebeQC}. He and Schramm predicted that there is a strong connection between the problem of rigidity and the problem of conformal removability, and formulated the following conjecture.

\begin{conjecture}A circle domain is conformally rigid if and only if every compact subset of its point boundary components is conformally removable. 
\end{conjecture}

Here a compact set $E$ is \textit{conformally removable} if every homeomorphism of the Riemann sphere $\widehat{\C}$ that is conformal in the complement of $E$ is necessarily a M\"obius transformation. The problem of characterizing conformally removable sets remains open. Recently, in an attempt to resolve this, the author has introduced the class of $\CNED$ sets, i.e., \textit{countably negligible for extremal distance}, which is a generalization of the classical $\NED$ sets, studied by Ahlfors--Beurling  \cite{AhlforsBeurling:Nullsets}. Roughly speaking, a set $E$ in the Riemann sphere is $\CNED$ if the conformal modulus of a curve family is not affected when one restricts to the subfamily intersecting $E$ at countably many points. It is shown in \cite{Ntalampekos:exceptional} that $\CNED$ sets are removable for conformal maps and that they include several classes of sets that were known to be removable, such as sets of $\sigma$-finite Hausdorff $1$-measure and boundaries of domains satisfying the above-mentioned quasihyperbolic condition. Conjecturally, $\CNED$ sets coincide with conformally removable sets.

Our first main result verifies the rigidity conjecture for the class of $\CNED$ sets.

\begin{theorem}[Rigidity]\label{theorem:rigidity}
A circle domain is conformally rigid if every compact subset of its point boundary components is $\CNED$. 
\end{theorem}

This result is the strongest so far, implying the rigidity results of He--Schramm \cites{HeSchramm:Uniformization,HeSchramm:Rigidity} and the joint result of Younsi with the author \cite{NtalampekosYounsi:rigidity}. Moreover, it provides substantial evidence for the conjecture and its proof is much more conceptual compared to proofs of previous results. The proof is given in Section \ref{section:rigidity}.

As in all previous works related to the rigidity conjecture, the first steps in the proof of Theorem \ref{theorem:rigidity} are to establish that a conformal homeomorphism between circle domains extends continuously and then homeomorphically to the closures of the domains under some conditions. The problem of continuous extension has its own interest and venerable history.

\begin{question}When does a conformal map between domains in the Riemann sphere extend continuously or homeomorphically to the closures of the domains?
\end{question} 

This fundamental question was first answered for simply connected domains by Carath\'eodory, who showed that a conformal map from the unit disk onto a simply connected domain $D$ extends continuously to the closures if and only if $\partial D$ is locally connected. Moreover, the extension is homeomorphic if and only $\partial D$ is a Jordan curve. This result extends with the same techniques to finitely connected domains bounded by Jordan curves. However, for domains of infinite connectivity it is well known that these results fail. 

Schramm \cite{Schramm:transboundary} studied the extension problem in  a class of domains resembling circle domains. Namely, he considered \textit{generalized Jordan domains}, that is, domains whose boundary components are Jordan curves or points with diameters shrinking to zero, and he imposed the assumption of \textit{cofatness} of the domains. A domain is cofat if each complementary component $P$ satisfies the estimate
$$ \area(P\cap B(x,r))\geq \tau r^2$$
for every ball $B(x,r)$ centered at a point of $P$ that does not contain $P$. Here $\tau>0$ is a uniform constant. In modern terminology, $P$ is Ahlfors $2$-regular.  Specifically, Schramm showed that a conformal map between countably connected, cofat, generalized Jordan domains extends to a homeomorphism of the closures.

He--Schramm \cite{HeSchramm:Rigidity} later developed further these techniques, while approaching the rigidity conjecture, and showed that a conformal map $f\colon U\to V$ between cofat generalized Jordan domains such that $\partial U$ has $\sigma$-finite Hausdorff $1$-measure extends to a homeomorphism of the closures. This is the first non-trivial result beyond the countably connected case.  Younsi and the author \cite{NtalampekosYounsi:rigidity} developed techniques that establish the same result for cofat generalized Jordan domains such that the quasihyperbolic metric on $U$ satisfies a certain integrability condition that is discussed above. 

It is  quite remarkable that the same conditions appear in three different problems: conformal rigidity, conformal removability, and continuous extension.  Our second main theorem reveals the connection of $\CNED$ sets to the problem of continuous extension and puts in common framework all previous results.

\begin{theorem}[Extension]\label{theorem:continuous_homeo}
Let $f\colon U\to V$ be a conformal homeomorphism between  generalized Jordan domains $U,V\subset \widehat\C$ such that  $U$ is cofat and every compact subset of the point components of $\partial U$ is $\CNED$. 
\begin{enumerate}[\upshape(i)]
	\item\label{theorem:continuous:i} If the diameters of the complementary components of $V$ lie in $\ell^2$, then $f$ has an extension to a continuous, surjective, and monotone map $F\colon \widehat{\C}\to \widehat{\C}$ such that $F^{-1}(\br V)=\br U$. 
	\item\label{theorem:continuous:ii} If $V$ is cofat, then $f$ has an extension to a homeomorphism $F\colon \widehat{\C}\to \widehat{\C}$.
\end{enumerate}
\end{theorem}

Here a map is \textit{monotone} if the preimage of every point is a continuum. Moreover, continuous, surjective, and monotone maps of the sphere as in \ref{theorem:continuous:i} are precisely uniform limits of homeomorphisms \cite{Youngs:monotone}.

Our result also improves on a recent result of Luo--Yao \cite[Theorem 1]{LuoXiaoting:continuous}, who obtain a continuous extension under the much more restrictive assumptions that the point components of $\partial U$ have $\sigma$-finite Hausdorff $1$-measure and the diameters of the complementary components of $V$ lie in $\ell^1$, rather than in $\ell^2$. However, their results extend to domains beyond circle domains and beyond the cofat category. Our techniques extend to that setting as well; in addition, one can slightly relax the assumptions in Theorem \ref{theorem:continuous_homeo}, imposing no assumptions on finitely many complementary components of $U,V$. We do not discuss these generalizations here as they are not related to the rigidity conjecture.

The proof of Theorem \ref{theorem:continuous_homeo} is given in Section \ref{section:continuous}. It relies on the notions of Sierpi\'nski packings and packing-conformal maps that were recently introduced by the author in \cite{Ntalampekos:uniformization_packing}. We first establish in Section \ref{section:upgrade} that conformal maps as in Theorem \ref{theorem:continuous_homeo} satisfy a certain \textit{transboundary upper gradient inequality}. This inequality is then combined with results from \cite{Ntalampekos:uniformization_packing} to give Theorem \ref{theorem:continuous_homeo}. Without further assumptions, the extension provided by \ref{theorem:continuous:i} is not a homeomorphism. We provide an example in Section \ref{section:example}.

\begin{prop}\label{proposition:example}
There exists a conformal map $f$ from a circle domain $U$ onto a generalized Jordan domain $V$ such that the diameters of the complementary components of $V$ lie in $\ell^2$ and $f$ does not extend to a homeomorphism of the sphere.
\end{prop}

\section{Preliminaries}

Throughout the paper we will use the spherical metric $\sigma$ and spherical measure $\Sigma$ on $\widehat{\C}=\C\cup \{\infty\}$. The Euclidean metric on $\C$ is denoted by $e$. In a metric space the open ball with center $x$ and radius $r$ is denoted by $B(x,r)$. If $B=B(x,r)$ and $k>0$, then $kB$ denotes the ball $B(x,kr)$. We use the notation $N_r(A)$ for the open $r$-neighborhood of a set $A$. 

The \textit{$1$-dimensional Hausdorff measure} $\mathcal H^1(E)$ of a set $E\subset \widehat{\C}$ is defined by
$$\mathcal{H}^{1}(E)=\lim_{\delta \to 0} \mathcal{H}_\delta^{1}(E)=\sup_{\delta>0} \mathcal{H}_\delta^{1}(E),$$
where 
$$
\mathcal{H}_\delta^{1}(E)=\inf \left\{\sum_{j} \operatorname{diam}(U_j): E \subset \bigcup_j U_j,\, \operatorname{diam}(U_j)<\delta \right\}
$$
for $\delta\in [0,\infty]$. When $\delta=\infty$, the latter quantity is called the $1$-dimensional Hausdorff content of $E$. The Hausdorff $1$-content is an outer measure on $\widehat{\C}$. We always have trivially
\begin{align*}
\min \{\mathcal H^1(E),\diam(E)\}\geq \mathcal H^1_{\infty}(E).
\end{align*}
Moreover, if $E$ is connected, then 
\begin{align*}
\mathcal H^1_{\infty}(E) = \diam(E). 
\end{align*}
This can be proved by following the argument in \cite[Lemma 2.6.1, p.~53]{BuragoBuragoIvanov:metric}. 

Let $\tau>0$.  A measurable set $K\subset \widehat{\C}$ is \textit{$\tau$-fat} if for each $x\in K$ and for each ball $B(x,r)$ that does not contain $K$ we have $\Sigma( B(x,r)\cap K) \geq \tau r^2$. A set is \textit{fat} if it is $\tau$-fat for some $\tau>0$. Note that points are automatically $\tau$-fat for every $\tau>0$. A more modern terminology for fatness Ahlfors $2$-regularity. See \cite[Lemma 2.7]{Ntalampekos:uniformization_packing} for a proof of the next lemma.

\begin{lemma}\label{lemma:bojarski}
Let $a\geq 1$ and $\{b_i\}_{i\in I}$ be a collection of non-negative numbers. Suppose that $\{D_i\}_{i\in I}$ is a family of measurable sets and $\{B_i=B(x_i,r_i)\}_{i\in I}$ is a family of balls in $\widehat{\C}$ such that $D_i\subset B_i$ and $\Sigma(B_i)\leq a \Sigma(D_i)$ for each $i\in I$. Then
\begin{align*}
\left \| \sum_{i\in I} b_i\chi_{B_i}\right \|_{L^2(\widehat{\C})} \leq c(a)  \left \| \sum_{i\in I} b_i\chi_{D_i} \right \|_{L^2(\widehat{\C})}
\end{align*}
\end{lemma}

\subsection{\texorpdfstring{\textit{CNED} sets}{CNED sets}}\label{section:cned}

A \textit{curve} or \textit{path} in $\widehat{\C}$ is a continuous function $\gamma\colon I\to \widehat{\C}$, where $I$ is an interval. The \textit{trace} of $\gamma$ is the set $|\gamma|=\gamma(I)$. A path  is \textit{simple} if it is injective.

We give the definition of $2$-modulus of a curve family in the sphere. Let $\Gamma$ be a family of curves in $\widehat{\C}$. We say that a Borel function $\rho\colon \widehat{\C}\to[0,\infty]$ is \textit{admissible} for the curve family $\Gamma$ if
$$\int_{\gamma}\rho\, ds\geq 1$$
for each locally rectifiable curve $\gamma\in \Gamma$. We then define the \textit{$2$-modulus} or else \textit{conformal modulus} of $\Gamma$ as
$$\md_2\Gamma=\inf_{\rho}\int \rho^2 \, d\Sigma,$$
the infimum taken over all admissible functions $\rho$. 

For a set $E\subset \widehat{\C}$ we denote by $\mathcal F_{\sigma}(E)$ the family of curves $\gamma$ intersecting the set $E$ at countably many points; that is, the set $E\cap |\gamma|$ is countable. For two sets $F_1,F_2\subset \widehat{\C}$ we denote by $\Gamma(F_1,F_2)$ the family of curves $\gamma\colon [a,b]\to \widehat{\C}$ with  $\gamma(a)\in F_1$ and $\gamma(b)\in F_2$. A set $E\subset \widehat{\C}$ is $\CNED$ if for every pair of non-empty, disjoint continua $F_1,F_2\subset \widehat{\C}$ we have
\begin{align*}
\md_2\Gamma(F_1,F_2)= \md_2(\Gamma(F_1,F_2)\cap \mathcal F_{\sigma}(E)).
\end{align*}
We remark that $\CNED$ sets are not assumed to be compact. In \cite{Ntalampekos:exceptional} $\CNED$ sets are assumed to be subsets of $\C$ but one can also work more generally with subsets of $\widehat{\C}$ using the conformal invariance of modulus and the fact that the spherical and Euclidean metrics on $\C\subset \widehat{\C}$ are conformally equivalent. We list several properties of $\CNED$ sets.

\begin{enumerate}[label=(C\arabic*)]
	\item\label{c:rect} If $\mathcal H^1(E)<\infty$, then $E$ is $\CNED$ \cite[Theorem 1.5]{Ntalampekos:exceptional}.
	\smallskip  
	\item\label{c:meas} Every measurable $\CNED$ set has measure zero \cite[Lemma 2.5]{Ntalampekos:exceptional}.
	\smallskip
	\item\label{c:bilip} If $E\subset \widehat{\C}$ is a compact $\CNED$ set and $f$ is a bi-Lipschitz embedding from a neighborhood of $E$ into $\widehat{\C}$, then $f(E)$ is $\CNED$ \cite[Corollary 7.2]{Ntalampekos:exceptional}.
	\smallskip
	\item\label{c:union} A countable union of compact $\CNED$ sets is $\CNED$ \cite[Theorem 1.7]{Ntalampekos:exceptional}. 
\end{enumerate}

The next theorem is a special case of \cite[Theorem 7.1 (V)]{Ntalampekos:exceptional}, which gives a characterization of compact $\CNED$ sets.

\begin{theorem}\label{theorem:cned}
Let $E\subset \widehat{\C}$ be a compact $\CNED$ set. Then for each Borel function $\rho\colon \widehat{\C}\to [0,\infty]$ with $\rho\in L^2(\widehat{\C})$ there exists a path family $\Gamma_0$ with $\md_2\Gamma_0=0$ satisfying the following statements. For every rectifiable path $\gamma\colon [a,b]\to \widehat{\C}$ outside $\Gamma_0$ with  $\gamma(a),\gamma(b)\notin E$ and $\gamma(a)\neq \gamma(b)$ and for every $\varepsilon>0$ there exists a rectifiable simple path  $\widetilde \gamma$ with the same endpoints as $\gamma$ and a finite collection of non-constant paths $\{\gamma_i\}_{i\in I}$ such that the following statements are true.
\begin{enumerate}[\upshape(i)]
	\item The set $E\cap |\widetilde \gamma|$ is countable.
	\item\label{cned:ii} $|\widetilde \gamma| \subset |\gamma|\cup \bigcup_{i\in I}|\gamma_i|\subset  N_{\varepsilon}(|\gamma|)$.
	\item $\sum_{i\in I} \ell(\gamma_i)<\varepsilon$.
	\item $\sum_{i\in I} \int_{\gamma_i}\rho\, ds <\varepsilon.$
\end{enumerate}
Moreover the paths $\gamma_i$, $i\in I$, can be taken lie outside a given path family $\Gamma_1$ with $\md_2\Gamma_1=0$. 
\end{theorem}

$\CNED$ sets play an important role in the theory of quasiconformal maps, as they are removable in the following sense.

\begin{theorem}[{\cite[Corollary 1.4]{Ntalampekos:exceptional}}]\label{theorem:cned:removable}
Let $E\subset \widehat{\C}$ be a compact $\CNED$ set. Then every homeomorphism $f\colon U\to V$ between open sets $U,V\subset 
\widehat{\C}$ that is (quasi)\-conformal on $U\setminus E$ is (quasi)conformal on $U$.  
\end{theorem}

We also have a more general statement for sets $E$ that are not compact. We define the \textit{eccentricity} of an open set $A$ in a metric space $(X,d)$  by
$$E_d(A)= \inf\{M\geq 1: \textrm{there exists an open ball}\,\, B\subset A\subset MB\},$$
Suppose that $A_n$ is a sequence of open sets containing $x$ with $\diam(A_n)\to 0$. If $x\in \C$, then $E_{e}(A_n)= (1+o(1))E_{\sigma}(A_n)$ as $n\to\infty$. This allows us to drop the subscript $\sigma$ or $e$ from the eccentricity of sets. Moreover, if $f$ is a map between open subsets of the sphere that is conformal in a neighborhood of $x$, then $E(A_n)=(1+o(1))E(f(A_n))$ as $n\to\infty$. 

Let $f\colon U\to V$ be a homeomorphism between open subsets of $\widehat{\C}$. The \textit{eccentric distortion of $f$} at a point $x\in U$, denoted by $E_f(x)$, is the infimum of all values $H\geq 1$ such that there exists a sequence of open sets $A_n\subset U$, $n\in \N$, containing $x$ with $\diam(A_n)\to 0$ as $n\to\infty$ and with the property that $E(A_n)\leq H$ and $E(f(A_n))\leq H$ for each $n\in \N$.   By the above, the eccentric distortion $E_f$ is invariant under pre- and post-compositions of $f$ by a conformal or anti-conformal map. We now state the main theorem of \cite{Ntalampekos:exceptional}. 

\begin{theorem}[{\cite[Theorem 1.2]{Ntalampekos:exceptional}}]\label{theorem:cned:removable:eccentric}
Let $E\subset \widehat{\C}$ be a $\CNED$ set and let $f\colon U\to V$ be an orientation-preserving homeomorphism between open sets $U,V\subset \widehat{\C}$. If there exists $H\geq 1$ such that $E_f(x)\leq H$ for each point $x\in U\setminus E$, then $f$ is quasiconformal, quantitatively.  
\end{theorem}

Moreover, if $E_f=1$ a.e., then $f$ is conformal, as expected. We show this in the next lemma; see also \cite[Lemma 6.1]{BonkKleinerMerenkov:schottky} for a similar statement.

\begin{lemma}\label{lemma:conformal}
Let $f\colon U\to V$ be a quasiconformal homeomorphism between open sets $U,V\subset \widehat{\C}$ such that $E_f(x)=1$ for a.e.\ $x\in U$. Then $f$ is conformal. 
\end{lemma}

\begin{proof}
By post-composing $f$ with an isometry of  $\widehat{\C}$, we assume that we have a quasiconformal map $g\colon U'\to V'$ between planar domains with the property that for a.e.\ $x\in U'$ there exists a sequence of open sets $A_n$, $n\in \N$, containing $x$ and shrinking to $x$ such that $E(A_n)\to 1$ and $E(g(A_n))\to 1$ as $n\to\infty$. It suffices to show that $g$ is conformal. 

Let $x\in U'$ be a point of differentiability of $g$. We may assume that $x=g(x)=0$. We let $r_n=\diam(A_n)$, which converges to $0$ as $n\to\infty$, and define 
$$g_n(y)= r_n^{-1}g(r_ny)$$
in a neighborhood of $0$. As $n\to\infty$, this map converges locally uniformly in $\C$ to the linear map $Dg(0)$.  Since $E(A_n)=E(r_n^{-1}A_n)\to 1$, the sets $r_n^{-1}A_n$ converge in the Hausdorff sense to a non-degenerate closed ball containing $0$.  Therefore, the sets
$$g_n(r_n^{-1}A_n)= r_n^{-1}g(A_n)$$
converge in the Hausdorff sense to a compact set containing $0$. Since $E(r_n^{-1}g(A_n))=E(g(A_n))\to 1$, we conclude that this compact set is a possibly degenerate closed ball. Every linear transformation that maps a non-degenerate ball to a ball is conformal. This implies that $\|Dg\|=J_g$ a.e., so $g$ is conformal on $U'$. 
\end{proof}

\subsection{Sierpi\'nski packings and packing-quasiconformal maps}

Let $\{p_i\}_{i\in \N}$ be a collection of pairwise disjoint, non-separating continua in $\widehat{\C}$ such that $\diam(p_i)\to 0$ as $i\to\infty$. The collection $\{p_i\}_{i\in \N}$ is called a Sierpi\'nski packing and the set  $X=\widehat{\C}\setminus \bigcup_{i\in \N} p_i$ is its {residual set}. When there is no confusion, we call $X$ a Sierpi\'nski packing and the underlying collection $\{p_i\}_{i\in \N}$ is implicitly understood. The continua $p_i$, $i\in \N$, are  called the {peripheral continua} of $X$. A Sierpi\'nski packing (resp.\ domain) is \textit{cofat} if there exists $\tau>0$ such that each of its peripheral continua (resp.\ complementary components) is $\tau$-fat.

Let $X=\widehat{\C}\setminus \bigcup_{i\in I} p_i$ be a Sierpi\'nski packing or a domain, where in the latter case the collection $\{p_i\}_{i\in I}$ is understood to comprise the complementary components. We  consider the quotient space $\mathcal E(X)=\widehat{\C}/\{p_i\}_{i\in I}$, together with the natural projection map $\pi_X\colon \widehat{\C} \to \mathcal E(X)$. We use the notation $\widehat{E}=\pi_X(E)$ for sets $E\subset \widehat{\C}$. The decomposition of $\widehat{\C}$ into the singleton points of $X$ and the continua $p_i$, $i\in I$, is always upper semicontinuous. Therefore, Moore's theorem \cite{Moore:theorem} implies that the space $\mathcal E(X)$ is a topological $2$-sphere. See \cite[Section 2] {Ntalampekos:uniformization_packing} for a further discussion.

Following \cite{Ntalampekos:uniformization_packing}, for two Sierpi\'nski packings or domains $X,Y$, we define the notion a packing-quasi\-conformal map between the associated topological spheres $\mathcal E(X),\mathcal E(Y)$ as follows.

\begin{definition}\label{definition:packing_quasiconformal}
Let   $X=\widehat \C\setminus \bigcup_{i\in I} p_i$ and $Y=\widehat \C \setminus \bigcup_{i\in I} q_i$ be Sierpi\'nski packings or domains. Let $h\colon \mathcal E(X)\to\mathcal E(Y)$ a continuous, surjective, and monotone map such that $h(\widehat{p}_i)=\widehat q_i$ for each $i\in I$. We say that $h$ is \textit{packing-quasiconformal} if there exists $K\geq 1$ and a non-negative Borel function $\rho_h\in  L^2(\widehat{\C})$ with the following properties.
	\begin{enumerate}[\upshape(i)]
		\item\label{qc:i} \textit{(Transboundary upper gradient inequality)} There exists a curve family $\Gamma_0$ in $\widehat{\C}$ with $\md_2\Gamma_0=0$ such that for all curves $\gamma\colon [a,b]\to \widehat \C$ outside $\Gamma_0$ we have
	\begin{align*}
\dist( \pi_Y^{-1} \circ h \circ \pi_X(\gamma(a)), \pi_Y^{-1}\circ h\circ \pi_X(\gamma(b)))\leq \int_{\gamma}\rho_h\, ds + \sum_{i:p_i\cap |\gamma|\neq \emptyset} \diam(q_i).
\end{align*}
		\item\label{qc:ii} \textit{(Quasiconformality)} For each Borel set $E\subset \widehat \C$ we have
	$$\int_{\pi_X^{-1}(h^{-1}(\pi_Y(E)))} \rho_h^2\, d\Sigma \leq K\Sigma  (E\cap Y).$$ 
\end{enumerate}
In this case, we say that $h$ is \textit{packing-$K$-quasiconformal}. If $K=1$, then $h$ is called \textit{packing-conformal}.
\end{definition}

The following theorem summarizes Theorems 6.1 and 7.1 from \cite{Ntalampekos:uniformization_packing}. This will be the main tool for establishing Theorem \ref{theorem:continuous_homeo}.

\begin{theorem}\label{theorem:continuous_extension_packing_qc}
Let $X=\widehat{\C}\setminus \bigcup_{i\in \N}p_i$ and $Y=\widehat{\C}\setminus \bigcup_{i\in \N} q_i$ be Sierpi\'nski packings such that the peripheral continua of $X$ are uniformly fat closed Jordan regions or points and the peripheral continua of $Y$ are closed Jordan regions or points with diameters lying in $\ell^2(\N)$. Let $ h\colon \mathcal E(X)\to \mathcal E(Y)$ be a packing-$K$-quasiconformal map for some $K\geq 1$.  Then there exists a continuous, surjective, and monotone map $H\colon \widehat \C \to \widehat \C$ such that $\pi_Y\circ H=h\circ \pi_X$ and $H^{-1}(\inter(q_i))=\inter(p_i)$ for each $i\in \N$.  If, in addition, $Y$ is cofat, then $H$ may be taken to be a homeomorphism of the sphere. 
\end{theorem}

\section{Transboundary upper gradient inequality of conformal maps}\label{section:upgrade}

For a domain $U\subset \widehat{\C}$ we denote by $\mathcal C(U)$ the family of complementary components of $U$.  Let $f\colon U\to V$ be a homeomorphism between domains $U,V\subset \widehat\C$. Then $f$ induces a homeomorphism $\widehat{f}\colon \mathcal E(U)\to \mathcal E(V)$ such that $\pi_V\circ f= \widehat{f}\circ \pi_U$ on $U$. We define the set function $$f^*= \pi_V^{-1}\circ \widehat{f} \circ \pi_U$$ from subsets of $\widehat{\C}$ to subsets of $\widehat{\C}$. We note that if $p\in \mathcal C(U)$, then $f^*(p)$ is precisely the complementary component of $V$ with the property that $\{f(z_n)\}_{n\in \N}$ accumulates at $f^*(p)$ whenever $\{z_n\}_{n\in \N}$ is a sequence in $U$ accumulating at $p$. Furthermore, $f^*$ has the following properties.
\begin{enumerate}[(F1)]
	\item\label{property:continuum} For each continuum $E\subset \widehat{\C}$, $f^*(E)$ is a continuum.
	\item\label{property:convergence} If $E_n$, $n\in \N$,  is a sequence of compact sets in $\widehat{\C}$ with $E_{n+1}\subset E_n$, $n\in \N$, converging in the Hausdorff sense to a compact set $E$, then $f^*(E_n)$ converges to $f^*(E)$. 
\end{enumerate}
The first property relies on the fact that $\pi_V$ is a monotone map and the second follows from the continuity of $\pi_V$,$\widehat{f}$, and $\pi_U$. See Sections 2.3--2.5 from \cite{Ntalampekos:uniformization_packing} for more background. For a conformal homeomorphism $f\colon U\to V$ we consider the derivative $|Df|\colon U\to (0,\infty)$ in the Riemannian metric of $\widehat{\C}$. If $z,f(z)\in \C$, then
\begin{align*}
|Df|(z)= \frac{1+|z|^2}{1+|f(z)|^2}|f'(z)|.
\end{align*}
The first step of the proof of Theorem \ref{theorem:continuous_homeo} is to show that conformal maps as in the statement satisfy a transboundary upper gradient inequality, as stated in the next theorem.

\begin{theorem}\label{theorem:absolute_continuity}
Let $f\colon U\to V$ be a conformal homeomorphism between  generalized Jordan domains $U,V\subset \widehat\C$ such that  $U$ is cofat, every compact subset of the point components of $\partial U$ is $\CNED$, and the diameters of the complementary components of $V$ lie in $\ell^2$. Then there exists a family of curves $\Gamma_0$ in $\widehat \C$ with $\md_2\Gamma_0=0$ such that for all curves $\gamma\colon [a,b]\to \widehat{\C}$ outside $\Gamma_0$ with $\gamma(a),\gamma(b)\in U$, we have
\begin{align*}
\sigma( f(\gamma(a)),f(\gamma(b)))\leq \int_{\gamma} |Df|\chi_U\, ds+ \sum_{\substack{p\in \mathcal C(U)\\ p\cap |\gamma|\neq \emptyset}} \diam(f^*(p)).
\end{align*}
\end{theorem}

The exclusion of a family $\Gamma_0$ of modulus zero is necessary. Indeed, if $E$ is a totally disconnected subset of a geodesic curve $\gamma$ such that $\mathcal H^1(E)>0$ and $f$ is the identity map on $U=V=\widehat{\C}\setminus E$, then the transboundary upper gradient inequality fails along $\gamma$.

\begin{proof}
Let $\{p_i\}_{i\in I}$ be the collection of non-degenerate complementary components of $U$. Since $U$ is a generalized Jordan domain, $I$ is a countable set and  we may assume that $I\subset \N$. We also set $q_i=f^*(p_i)$, $i\in I$. We define $\rho=|Df|\chi_U$.  We will show that there exists a curve family $\Gamma_0$ of $2$-modulus zero such that for all rectifiable curves $\gamma\colon [a,b]\to \widehat{\C}$ outside $\Gamma_0$ with $\gamma(a),\gamma(b)\in U$ and $\gamma(a)\neq \gamma(b)$ we have
\begin{align}\label{proposition:qc:transboundary}
	\sigma(f(\gamma(a)), f(\gamma(b))) \leq \int_{\gamma}\rho\, ds + \sum_{i:p_i\cap |\gamma|\neq \emptyset}\diam(q_i).  
\end{align}
For the convenience of the reader, we split the proof into sections. 

\smallskip
\noindent
\textbf{Enlarging the complementary components.}
We fix $n\in \N$. By property \ref{property:convergence}, for each $i\in I$ we can find a Jordan region $W_i(n)$ such that
\begin{align}
\begin{aligned}\label{proposition:qc:diameters}&p_i\subset W_i(n)\subset \br{W_i(n)}\subset N_{\diam(p_i)/n}(p_i)\quad \textrm{and}\\
&q_i\subset f^*(W_i(n))\subset f^*(\br{W_i(n)}) \subset N_{\diam(q_i)/n}(q_i).
\end{aligned}
\end{align} 
Moreover, we require that 
$$W_{i}(n+1)\subset W_i(n)$$
for each $i\in I$ and $n\in \N$. Observe that $\diam(W_i(n))\to 0$ as $i\to\infty$, since $\diam(p_i)\to 0$. Also, $W_{i}(n)$ converges to $p_i$ as $n\to\infty$ in the Hausdorff sense.

For each  $i\in I$ consider a ball $B_{i,n}$ centered at a point of $W_i(n)$ and with radius $\diam(W_i(n))$. 
Consider the function 
$$h_{n} =\sum_{i\in I} \frac{\diam(q_i)}{\diam(W_i(n))}\chi_{2B_{i,n}}.$$
The uniform fatness of $p_i$ implies that
\begin{align*}
\Sigma(2B_{i,n}) \lesssim \diam(W_i(n))^2\lesssim (1+2/n)^2\diam(p_i)^2\lesssim \Sigma(p_i)
\end{align*}
for each $i\in I$. By Lemma \ref{lemma:bojarski} and the fact that the sets $p_i$, $i\in I$, are disjoint, we have
\begin{align}
\begin{aligned}\label{proposition:qc:l2}
\|h_n\|_{L^2(\widehat{\C})}^2 &\lesssim \int\left(\sum_{i\in I} \frac{\diam(q_i)}{\diam(W_i(n))}\chi_{p_i}\right)^2 \,d\Sigma \simeq \sum_{i\in I} \left( \frac{\diam(q_i)^2}{\diam(W_i(n))^2} \Sigma(p_i)\right)\\
&\lesssim \sum_{i\in I}\diam(q_i)^2<\infty.
\end{aligned}
\end{align}
This implies that there exists a path family $\Gamma_1(n)$ of $2$-modulus zero such that 
$$\int_{\gamma}h_n\, ds <\infty$$
for each $\gamma\notin \Gamma_1(n)$. We now fix a non-constant path $\gamma\notin \Gamma_1(n)$.  Let $J$ be the set of indices $i\in I$ such that  $\br{W_i(n)}\cap |\gamma|\neq\emptyset$ and $\gamma$ is not contained in $2B_{i,n}$. Then
\begin{align*}
\int_{\gamma} \chi_{2B_{i,n}}\, ds \geq \diam( W_{i}(n))
\end{align*}
for $i\in J$. Thus, 
$$\sum_{i\in J} \diam(q_i) \leq \int_{\gamma} \sum_{i\in J } \frac{\diam(q_i)}{\diam(W_i(n))}\chi_{2B_{i,n}}\, ds \leq \int_{\gamma} h_{n}\, ds <\infty. $$
Since $\gamma$ is non-constant and $\diam(2B_{i,n})\to 0$ as $i\to\infty$, there exist at most finitely many $i\in I\setminus J$. Therefore, for non-constant paths $\gamma\notin \Gamma_1(n)$ we have
\begin{align}\label{proposition:qc:h}
\sum_{i: \br{W_i(n)} \cap |\gamma|\neq \emptyset}\diam(q_i) <\infty.
\end{align}

\smallskip
\noindent
\textbf{Preliminary reductions.} We will show that for each $n\in \N$ there exists a curve family $\Gamma_0(n)$ of $2$-modulus zero such that for all rectifiable curves $\gamma\colon [a,b]\to \widehat{\C}$ outside $\Gamma_0(n)$ with $\gamma(a),\gamma(b)\in U$ and $\gamma(a)\neq \gamma(b)$ we have
\begin{align}\label{proposition:qc:transboundary1}
	\sigma(f(\gamma(a)), f(\gamma(b))) \leq \int_{\gamma}\rho\, ds +(1+2/n)\cdot  \sum_{i:\br{W_i(n)}\cap |\gamma|\neq \emptyset}\diam(q_i).
\end{align}
We now show how this implies \eqref{proposition:qc:transboundary}. We define 
$$\Gamma_0 =\Gamma_1(1)\cup \bigcup_{n\in \N}\Gamma_0(n),$$
which has $2$-modulus zero. If $\gamma\notin \Gamma_0$, then \eqref{proposition:qc:transboundary1} holds for each $n\in \N$. Recall that $p_i\subset \br{W_{i}(n+1)}\subset \br{W_i(n)}$ for each $i\in I$, $n\in \N$, and  $\br{W_i(n)}\to p_i$ as $n\to\infty$. This implies that the index set $\{i\in I:\br{W_i(n)}\cap |\gamma|\neq \emptyset\}$ contains $\{i\in I: p_i\cap |\gamma|\neq \emptyset\}$ and decreases to that set as $n\to \infty$. Moreover, for each $n\in \N$, the set $\{i\in I:\br{W_i(n)}\cap |\gamma|\neq \emptyset\}$ is contained in $\{i\in I:\br{W_i(1)}\cap |\gamma|\neq \emptyset\}$. Since $\gamma\notin \Gamma_1(1)$, by \eqref{proposition:qc:h} we have
\begin{align*}
\sum_{i:\br{W_i(1)}\cap |\gamma|\neq \emptyset}\diam(q_i) <\infty.
\end{align*}
The dominated convergence theorem implies that 
\begin{align*}
\sum_{i:\br{W_i(n)}\cap |\gamma|\neq \emptyset}\diam(q_i) \to \sum_{i:p_i\cap |\gamma|\neq \emptyset}\diam(q_i)
\end{align*}
as $n\to\infty$. Hence, by taking limits in \eqref{proposition:qc:transboundary1} we obtain \eqref{proposition:qc:transboundary}.

\smallskip
\noindent
\textbf{Length bounds.}
From now on, we fix $n\in \N$ throughout the remainder of the proof and we focus on establishing \eqref{proposition:qc:transboundary1}.  We also set $W_i=W_{i}(n)$ and $B_i=B_{i,n}$, $i\in I$.  Consider the set $E=(\widehat{\C}\setminus U)\setminus \bigcup_{i\in I} W_i$. This is a compact set of point boundary components of $U$, so it is $\CNED$ by assumption. For a path $\gamma\colon [a,b]\to \widehat{\C}$ we denote by $I(\gamma)$ the family of indices $i\in I$ such that $\br {W_i}\cap |\gamma|\neq \emptyset$. Note that the set $E\cup \bigcup_{i\in I(\gamma)}\br {W_i} $ is compact. Indeed, $\diam(W_i)\to 0$ as $i\to \infty$, so each limit point $x$ of $\bigcup_{i\in I(\gamma)} \br{W_i}$ that lies outside that set is necessarily contained in $|\gamma|\cap (\widehat{\C}\setminus U)$. Since $x\in |\gamma|\setminus \bigcup_{i\in I(\gamma)}\br{W_i}= |\gamma|\setminus \bigcup_{i\in I}\br{W_i}$, we conclude that $x\in E$, as desired. 

The map $f$ is conformal, so for all rectifiable paths $\gamma$ in $U$ we have
$$\ell(f\circ \gamma)= \int_{\gamma} |Df|\, ds.$$
See \cite[Theorem 5.6, p.~14]{Vaisala:quasiconformal}. Let $\gamma\colon [a,b]\to \widehat{\C}$ be a rectifiable path. The set $(a,b)\setminus \gamma^{-1}(E \cup \bigcup_{i\in I(\gamma)} \br{W_i} )$ is open so it is a countable union of disjoint open intervals $O_j$, $j\in J$. Each path $\gamma|_{O_j}$ is rectifiable and is contained in $U$, so we have
\begin{align}
\begin{aligned}\label{proposition:qc:gamma1}
\mathcal H^1_{\infty}(f^*(|\gamma|\setminus (E \cup \bigcup_{i\in I(\gamma)} \br{W_i})))&\leq \sum_{j\in J} \mathcal H^1_{\infty}( f(\gamma(O_j))) \leq  \sum_{j\in J} \diam(f(\gamma(O_j)))\\
&\leq \sum_{j\in J} \ell( f\circ \gamma|_{O_j})=\sum_{j\in J} \int_{\gamma|_{O_j}} |Df|\, ds\\& \leq \int_{\gamma} \rho\, ds.
\end{aligned}
\end{align}

\smallskip
\noindent
\textbf{Tail bounds.}
For $m\in \N$, consider the function 
$$g_m =\sum_{i\in I, i\geq m} \frac{\diam(q_i)}{\diam(W_i)}\chi_{2B_i}\leq g_1=h_n.$$
Observe that for each $d>0$ there exists $N(d)>0$ such that every curve $\gamma$ with  $\diam(|\gamma|)\geq d$ is not contained in $2B_i$ whenever $i\geq N(d)$. In particular, 
\begin{align}\label{proposition:qc:rho_n}
\sum_{i\in I(\gamma)\cap I, i\geq m}\diam(q_i)\leq \int_{\gamma}g_m\, ds \quad \textrm{for}\quad m\geq N(d).
\end{align}
Recall that $g_1=h_n\in L^2(\widehat{\C})$ by \eqref{proposition:qc:l2} and $\int_{\gamma}g_1\, ds<\infty$ for all $\gamma\notin \Gamma_1$, where $\Gamma_1=\Gamma_1(n)$. In particular, for $\gamma\notin \Gamma_1$ we have 
\begin{align}\label{proposition:qc:gm}
\int_{\gamma}g_m\, ds= \sum_{i\in I,i\geq m} \int_{\gamma}\frac{\diam(q_i)}{\diam(W_i)}\chi_{2B_i}\, ds \to 0
\end{align}
as $m\to\infty$, because these are tails of a convergent series. 

\smallskip
\noindent
\textbf{Avoiding degeneracies.}
Consider the set of points components $p\in \partial U$ such that 
$f^*(p)$ is a non-degenerate element of $\mathcal C(V)$. Since $V$ is a generalized Jordan domain, this set is countable. We let $\Gamma_2$ be the family of non-constant curves intersecting that set. Then $\md_2 \Gamma_2=0$ \cite[\S 7.9, p.~23]{Vaisala:quasiconformal}. Observe that if $\gamma$ is a non-constant path outside $\Gamma_2$, then for each $x\in |\gamma|\setminus \bigcup_{i\in I}p_i$ the set $f^*(x)$ is a singleton. Therefore, if $|\gamma|\cap E$ is a countable set, then $f^*(|\gamma|\cap E)$ is also countable. 

\smallskip
\noindent
\textbf{\textit{CNED} condition.} We apply Theorem \ref{theorem:cned} to the $\CNED$ set $E$ and the function $\rho+g_1$ and fix a path family $\Gamma_3$ with $\md_2\Gamma_3=0$ as in the statement. Now, we set $\Gamma_0(n)=\Gamma_1\cup \Gamma_2\cup \Gamma_3$, which is a family of $2$-modulus zero.  For the remaining of the proof we fix a rectifiable path $\gamma\colon [a,b]\to \widehat{\C}$ outside $\Gamma_0(n)$ such that $\gamma(a),\gamma(b)\in U$ and $\gamma(a)\neq \gamma(b)$. Our goal is to show \eqref{proposition:qc:transboundary1}.

By Theorem \ref{theorem:cned}, for each $\delta>0$ there exists a finite collection of non-constant paths $\{\gamma_j\}_{j\in J}$ and a rectifiable simple path $\widetilde \gamma$ with the same endpoints as $\gamma$ such that $|\widetilde \gamma|\cap E$ is countable, 
\begin{align}\label{proposition:qc:contain}
|\widetilde \gamma|\subset |\gamma|\cup \bigcup_{j\in J} |\gamma_j|\subset N_{\delta}(|\gamma|),
\end{align}
and
\begin{align}\label{proposition:qc:delta}
\sum_{j\in J} \int_{\gamma_j}(\rho+g_1)\, ds<\delta.
\end{align}
Moreover, the non-constant paths $\gamma_j$, $j\in J$, can be taken to be outside the family $\Gamma_2$, by the last part of Theorem \ref{theorem:cned}.  The definition of $\Gamma_2$ and the fact that the curves $\gamma,\gamma_j$, $j\in J$, are outside $\Gamma_2$ imply that $\widetilde \gamma\notin \Gamma_2$.  Since $|\widetilde \gamma|\cap E$ is countable, $f^*(|\widetilde \gamma|\cap E)$ is also countable by the previous, so 
\begin{align}\label{proposition:qc:countable}
\mathcal H^1_{\infty}(f^*(|\widetilde \gamma|\cap E))=0.
\end{align}
Observe that if a set $\br{W_i}$ does not intersect $|\gamma|$, then it is also disjoint from $N_{\delta}(|\gamma|)$, and thus from $|\widetilde \gamma|$,  for all sufficiently small $\delta>0$. This implies that for each $m\in \N$ there exists $\delta_0(m)>0$ such that for $\delta <\delta_0(m)$ we have
\begin{align}\label{proposition:qc:index}
I(\widetilde \gamma)\setminus I(\gamma)&\subset \{ i\in I(\widetilde \gamma): i\geq m\}.
\end{align}

\smallskip
\noindent
\textbf{Completion of the proof.} We let $d=\sigma(\gamma(a),\gamma(b))>0$ and fix $N(d)>0$ such that inequality \eqref{proposition:qc:rho_n} is true for all $m\geq N(d)$ and for all curves with diameter larger than $d$. We fix $m\geq N(d)$, so for all $\delta<\delta_0(m)$ there exists a simple path $\widetilde \gamma$ satisfying the above conditions. Since $\widetilde \gamma$ has the same endpoints as $\gamma$, we have $\diam(|\widetilde \gamma|)\geq d$. We now apply  \eqref{proposition:qc:index}, \eqref{proposition:qc:rho_n}, \eqref{proposition:qc:contain}, and \eqref{proposition:qc:delta} to obtain
\begin{align*}
 \sum_{i\in I(\widetilde \gamma)\setminus I(\gamma)} \diam(q_i) &\leq \sum_{i\in I(\widetilde \gamma), i\geq m} \diam(q_i) \leq  \int_{\widetilde \gamma}g_m\, ds\\&\leq  \int_{\gamma}g_m\, ds + \sum_{j\in J} \int_{\gamma_j}g_m\, ds\leq \int_{\gamma}g_m\, ds + \sum_{j\in J} \int_{\gamma_j}g_1\, ds \\
 &< \int_{\gamma}g_m\, ds+ \delta.
\end{align*}
This inequality and \eqref{proposition:qc:diameters} imply that
\begin{align}
\begin{aligned}\label{proposition:qc:w_bounds}
\mathcal H^1_{\infty}(f^*(  \bigcup_{i\in I(\widetilde \gamma)}\br{W_i}))&\leq \sum_{i\in I(\widetilde \gamma)} \diam(f^*(\br{W_i}))\leq (1+2/n)\sum_{i\in I(\widetilde \gamma)} \diam(q_i)\\
&\leq (1+2/n)\left(\sum_{i\in I(\gamma)}\diam(q_i) + \int_{\gamma}g_m\, ds +\delta\right). 
\end{aligned}
\end{align}
Since $\widetilde \gamma$ is a simple rectifiable path, by \eqref{proposition:qc:countable}, \eqref{proposition:qc:gamma1}, \eqref{proposition:qc:contain}, and  \eqref{proposition:qc:delta} we have
\begin{align}
\begin{aligned}\label{proposition:qc:length_bounds}
\mathcal H^1_{\infty}(f^*(|\widetilde \gamma|\setminus \bigcup_{i\in I(\widetilde \gamma)} \br{W_i})) &=\mathcal H^1_{\infty}(f^*(|\widetilde \gamma|\setminus (E \cup \bigcup_{i\in I(\widetilde \gamma)} \br{W_i})))\\
&\leq  \int_{\widetilde \gamma}\rho\, ds\leq  \int_{\gamma}\rho\, ds +\sum_{j\in J} \int_{\gamma_j}\rho \, ds\\
&< \int_{\gamma}\rho\, ds+\delta.
\end{aligned}
\end{align}
Finally, by property \ref{property:continuum} the set $f^*(|\widetilde \gamma|)$ is a continuum containing $f(\gamma(a))$ and $f(\gamma(b))$. Therefore, the estimates \eqref{proposition:qc:w_bounds} and \eqref{proposition:qc:length_bounds} give
\begin{align*}
\sigma(f(\gamma(a)),f(\gamma(b))) &\leq \diam(f^*(|\widetilde \gamma|)) =\mathcal H^1_{\infty}(f^*(|\widetilde \gamma| )) \\ 
&\leq \mathcal H^1_{\infty}(f^*(|\widetilde \gamma|\setminus \bigcup_{i\in I(\widetilde \gamma)} \br{W_i}))  + \mathcal H^1_{\infty}(f^*(  \bigcup_{i\in I(\widetilde \gamma)}\br{W_i}))\\
&\leq \int_{\gamma}\rho\, ds + \delta + (1+2/n)\left(\sum_{i\in I(\gamma)}\diam(q_i) + \int_{\gamma}g_m\, ds+ \delta\right).
\end{align*}
This is true for each  $m\geq N(d)$ and for $\delta<\delta_0(m)$. We first let $\delta\to 0$. Then we let $m\to\infty$ and obtain $\int_{\gamma}g_m\, ds\to 0$ by \eqref{proposition:qc:gm}. This gives the desired \eqref{proposition:qc:transboundary1}.
\end{proof}

\section{Continuous and homeomorphic extension}\label{section:continuous}

\begin{proof}[Proof of Theorem \ref{theorem:continuous_homeo}]

Let $f\colon U\to V$ be a homeomorphism between generalized Jordan domains $U,V\subset \widehat{\C}$.  Let $\{p_i\}_{i\in I}$ be the collection of all non-degenerate components of $\widehat\C\setminus U$, together with the degenerate components that are mapped under $f^*$ to non-degenerate components of $\widehat{\C}\setminus V$. We define $q_i=f^*(p_i)$, $i\in I$.  In particular, $\{q_i\}_{i\in I}$ is the collection of non-degenerate complementary components of $V$ together with the images under $f^*$ of the non-degenerate complementary components of $U$. Note that the collection $I$ is countable, since $U$ and $V$ are generalized Jordan domains. 

If $p\in \mathcal C(U)\setminus \{p_i:i\in I\}$, then $p$ is a point and for every sequence $z_n\in U$ converging to $p$, the sequence $f(z_n)$ converges to the point $f^*(p)\in \mathcal C(V)\setminus \{q_i:i\in I\}$. This implies that $f$ extends to a continuous and injective map from the set $X=\widehat{\C}\setminus \bigcup_{i\in I}p_i$ onto $Y=\widehat{\C}\setminus \bigcup_{i\in I}q_i$. The same argument, applied to $f^{-1}$, shows that $f$ extends to a homeomorphism from $X$ onto $Y$. We denote this homeomorphism by $f$ as well.  Observe that $\br X=\br U$ and $\br Y=\br V$.

\smallskip
\noindent
\textit{Case 1: Finitely many non-degenerate components.} Suppose that the set $I$ is finite. Then $X$ and $Y$ are finitely connected domains. If $W$ is an open set such that $W\subset \br W\subset X$, then $E=\br W\setminus U$ is a compact subset of the point boundary components of $U$. By assumption, $E$ is $\CNED$. The fact that $f$ is conformal on $W\setminus E=W\cap U$ and Theorem \ref{theorem:cned:removable} imply that $f$ is conformal on $W$. Since $W$ was arbitrary, we conclude that $f$ is a conformal map from $X$ onto $Y$. It is a standard fact that a conformal map between finitely connected generalized Jordan domains extends to a homeomorphism of the closures; see \cite[Chapter I, \S 8]{LehtoVirtanen:quasiconformal}. In particular $f$ can be extended to a homeomorphism of $\widehat{\C}$. 

\smallskip
\noindent
\textit{Case 2: Infinitely many non-degenerate components.}
Suppose that the set $I$ is infinite. In this case, $X$ and $Y$ are Sierpi\'nski packings. Recall that $f$ induces a homeomorphism $\widehat{f}\colon \mathcal E(U)\to \mathcal E(V)$ such that $\pi_V\circ f= \widehat{f}\circ \pi_U$ on $U$. The spaces $\mathcal E(X)$ and $\mathcal E(U)$ are naturally identified because they are both obtained by collapsing the non-degenerate components of $\widehat \C\setminus U$ to points. Similarly $\mathcal E(Y)\equiv \mathcal E(V)$ and we have $\pi_Y \circ f=\widehat{f}\circ \pi_X$ on $X$. We will show that $\widehat{f}$ is packing-conformal, in the sense of Definition \ref{definition:packing_quasiconformal}. Observe first that $\widehat{f}(\pi_X(p_i))=\pi_Y(q_i)$ for each $i\in I$. 

We set $\rho= |Df|\chi_U$. For each Borel set $E\subset \widehat{\C}$, the conformality of $f$ implies that
\begin{align*}
\int_{f^{-1}(E\cap V)} \rho^2 \, d\Sigma =\Sigma( E\cap V)\leq \Sigma(E\cap Y).
\end{align*}
The function $\rho$ vanishes on the set $\pi_X^{-1} (\widehat f^{-1}(\pi_Y(E\setminus V)))$, since this set is disjoint from $U$. Also, $\pi_X^{-1} (\widehat f^{-1}(\pi_Y(E\cap V)))= f^{-1}(E\cap V)$. Thus, 
\begin{align*}
\int_{\pi_X^{-1} (\widehat f^{-1}(\pi_Y(E)))} \rho^2 \, d\Sigma \leq  \Sigma(E\cap Y).
\end{align*}
This verifies condition \ref{qc:ii} from Definition \ref{definition:packing_quasiconformal}.

Next, we verify  condition \ref{qc:i}. By Theorem \ref{theorem:absolute_continuity} there exists a curve family $\Gamma_0$ in $\widehat{\C}$ with $\md_2\Gamma_0=0$ such that for all curves $\gamma\colon [a,b]\to \widehat{\C}$ outside $\Gamma_0$ with $\gamma(a),\gamma(b)\in U$, we have
\begin{align*}
\sigma( f(\gamma(a)),f(\gamma(b)))\leq \int_{\gamma} \rho \, ds+ \sum_{\substack{i:p_i\cap |\gamma|\neq \emptyset}} \diam(q_i).
\end{align*}
By enlarging the exceptional curve family $\Gamma_0$, we assume that if $\gamma\notin \Gamma_0$, then all subpaths of $\gamma$ have the same property. The above inequality extends by continuity to paths $\gamma\notin \Gamma_0$ with $\gamma(a),\gamma(b)\in X$. Suppose that $\gamma(a)\notin X$ or $\gamma(b)\notin X$. If they both lie in $p_i$ for some $i\in I$, then $\pi_Y^{-1}\circ \widehat f\circ \pi_X(\gamma(a))=\pi_Y^{-1}\circ \widehat f\circ \pi_X(\gamma(b))=q_i$ and the transboundary upper gradient inequality as in \ref{qc:i} is trivial. Suppose that $\gamma(a)$ and $\gamma(b)$ do not lie in the same complementary component $p_i$, $i\in I$. Then there exists an open subpath $\gamma_1=\gamma|_{(a_1,b_1)}$ of $\gamma$ that does not intersect the components $p_i$, $i\in I$, that possibly contain $\gamma(a)$ or $\gamma(b)$ and the points $\gamma(a_1), \gamma(b_1)$ lie on the boundaries of the components that possibly contain $\gamma(a),\gamma(b)$, respectively. Arbitrarily close to $a_1$ (resp.\ $b_1$) we may find parameters $a_2>a_1$ (resp.\ $b_2<b_1$) such that $\gamma(a_2)\in U$ (resp.\ $\gamma(b_2)\in U$). This implies that 
\begin{align*}
\dist(\pi_Y^{-1}\circ \widehat f\circ \pi_X(\gamma(a)), \pi_Y^{-1}\circ \widehat f\circ \pi_X(\gamma(b)) ) &\leq \liminf_{\substack{a_2\to a_1^+ \\ b_2\to b_1^-}} \sigma( f(\gamma(a_2)),f(\gamma(b_2)) )\\
&\leq \int_{\gamma} \rho \, ds+ \sum_{\substack{i:p_i\cap |\gamma|\neq \emptyset}} \diam(q_i).
\end{align*}
This completes the proof that $\widehat{f}$ is packing-conformal. 

By Theorem \ref{theorem:continuous_extension_packing_qc}, there exists a continuous, surjective, and monotone map $F\colon \widehat \C \to \widehat \C$ such that $\pi_Y\circ F=\widehat f\circ \pi_X$ and $F^{-1}(\inter(q_i))=\inter(p_i)$ for each $i\in I$. In particular, we have $\pi_Y\circ F=\pi_Y\circ f$ on $X$. This implies that $F=f$ on $X\supset U$ and hence $F$ is an extension of $f$. Since $\br X=\br U$, $\br Y=\br V$,  and $F^{-1}(\br Y)=\br X$, we see that $F^{-1}(\br V)=\br U$, as desired. Finally, if $Y$ is cofat, then by the last part of Theorem \ref{theorem:continuous_extension_packing_qc}, $F$ may be taken additionally to be a homeomorphism of $\widehat{\C}$. 
\end{proof}

\section{Rigidity of circle domains}\label{section:rigidity}

The proof below relies on properties of $\CNED$ sets from Section \ref{section:cned}. 

\begin{proof}[Proof of Theorem \ref{theorem:rigidity}]
Let $f\colon U\to V$ be a conformal map. The first step is to extend $f$ to a homeomorphism of $\widehat{\C}$. This can be achieved through reflections across the boundary circles of $U$. A detailed proof can be found in \cite[Section 7.1]{NtalampekosYounsi:rigidity}; see also \cite[Section 5]{BonkKleinerMerenkov:schottky} for similar considerations. Here we highlight the important features of the extension procedure.  

We denote by $S_i$, ${i\in I}$, the collection of circles in $\partial U$, by $B_i\subset \widehat{\C}\setminus \br U$ the open ball bounded by $S_i$ and by $R_i$ the reflection across the circle $S_i$, $i\in I$. Here, we regard $I$ as a subset of $\N$.  Consider the free discrete group generated by the family of reflections $\{R_i\}_{i\in I}$. This is called the \textit{Schottky group} of $U$ and is denoted by $\Gamma(U)$. Each $T\in \Gamma(U)$ that is not the identity can be expressed uniquely as $T=R_{i_1}\circ \dots\circ R_{i_k}$, where $i_j\neq i_{j+1}$ for $j\in \{1,\dots, k-1\}$. We also note that $\Gamma(U)$ contains countably many elements.

By Theorem \ref{theorem:continuous_homeo}, $f$ extends to a homeomorphism between $\br U $ and $\br{V}$. Thus, there exists a natural bijection between $\Gamma(U)$ and $\Gamma(V)$, induced by $f$. Namely, if $R_i^*$ is the reflection across the circle $S_i^*=f(S_i)$, then for $T=R_{i_1}\circ\dots\circ R_{i_k}$ we define $T^*=R_{i_1}^*\circ \dots \circ R_{i_k}^*$. By \cite[Lemma 7.5]{NtalampekosYounsi:rigidity}, there exists a unique extension of $f$ to a homeomorphism $\widetilde f$ of $\widehat \C$ with the property that $T^*=\widetilde f\circ T\circ \widetilde{f}^{-1}$ for each $T\in \Gamma(U)$. We will verify that $\widetilde f$ is conformal. For simplicity, we use the notation $f$ instead of $\widetilde f$. 

For each point $x\in \widehat \C$ we have the following trichotomy; see Lemma 7.2 and Corollary 7.4 in \cite{NtalampekosYounsi:rigidity}.
\begin{enumerate}[\upshape(I)]
	\item\label{type:interior} (Interior type) $x\in T(U)$ for some $T\in \Gamma(U)$.
	\item\label{type:boundary} (Boundary type) $x\in T(\partial U)$ for some $T\in \Gamma(U)$.
	\item\label{type:buried} (Buried type) There exists a sequence of indices $\{i_j\}_{j\in \N}$ with $i_j\neq i_{j+1}$ and disks $D_0=B_{i_1}$, $D_k=R_{i_1}\circ\dots\circ R_{i_k}(B_{i_{k+1}})$ such that $D_{k+1}\subset D_{k}$ for each $k\geq 0$ and $\{x\}=\bigcap_{k=0}^\infty D_k$.
\end{enumerate}
At each point $x$ of interior type \ref{type:interior} the mapping $f$ is conformal, so it maps infinitesimal balls centered at $x$ to infinitesimal balls centered at $f(x)$; in particular $E_f(x)=1$. If  $x$ is of buried type \ref{type:buried}, then there exists a sequence of balls $D_k$, $k\in \N$, shrinking to $x$ such that $f(D_k)$, $k\in \N$, are balls shrinking to $f(x)$. It follows that $E_f(x)=1$. 

For points of boundary type \ref{type:boundary} we have a further distinction. For each $i\in I$ and $n\in \N$, we consider a Jordan region $W_i(n)$ such that  
\begin{align}\label{rigidity:w}
\br{B_i}\subset W_i(n) \subset (1+1/n)B_i\quad \textrm{and}\quad \br {f(B_i)}\subset f(W_i(n)) \subset (1+1/n)f(B_i).
\end{align}
We let $E_n=\partial U\setminus \bigcup_{i\in I} W_i(n)$, which is a compact subset of $\partial U$. Denote by $F$ the points of $\partial U$ that do not lie on $E=\bigcup_{n\in \N}E_n$ or on $S=\bigcup_{i\in I}S_i$. Thus, each point $x\in \widehat{\C}$ of boundary type \ref{type:boundary} satisfies one of the following conditions.
\begin{enumerate}[label=(II.\alph*)]
	\item\label{type:b:e} $x\in T(E)$ for some $T\in \Gamma(U)$.
	\item\label{type:b:s} $x\in T(S)$ for some $T\in \Gamma(U)$.
	\item\label{type:b:f} $x\in T(F)$ for some $T\in \Gamma(U)$.
\end{enumerate}

By assumption, $E_n$ is $\CNED$ for each $n\in \N$. Also, each $T\in \Gamma(U)$ is bi-Lipschitz, so by property \ref{c:bilip} in Section \ref{section:cned} the set $T(E_n)$ is also $\CNED$.  For each $T\in \Gamma(U)$ and $i\in I$ the circle $T(S_i)$ is rectifiable so it is $\CNED$ by \ref{c:rect}. Now \ref{c:union} implies that the set 
$$G=\bigcup_{T\in \Gamma(U)}T(E\cup S) = \bigcup_{T\in \Gamma(U)} \bigcup_{n\in \N} \bigcup_{i\in I} T(E_n\cup S_i)$$
is also $\CNED$, as a countable union of compact $\CNED$ sets. In particular the collection of points of boundary type \ref{type:b:e} and \ref{type:b:s} is $\CNED$. 

Next, we treat points of boundary type \ref{type:b:f}. Suppose that $F\neq \emptyset$. If $x\in F$, then $x\in \partial U$, $x$ does not lie on any circle of $\partial U$, and $x\in  \bigcup_{i\in I} W_i(n)$ for each $n\in \N$. Observe that for each $i\in I$ we can have $x\in W_{i}(n)$ only for finitely many $n\in \N$, since otherwise we have $x\in S_i$, a contradiction. We conclude that the index set $I$ is infinite, in which case we may assume that $I=\N$, and there exists a sequence $i_{n}\in I$ with $i_n\to \infty$ such that $x\in W_{i_n}(n)$. In particular, the regions $W_{i_n}(n)$ shrink to $x$. By \eqref{rigidity:w} we have $E(W_{i_n}(n))\leq 1+1/n$ and $E(f(W_{i_n}(n)))\leq 1+1/n$. We conclude that $E_f(x)=1$. If $x\in T(F)$ for some $T\in \Gamma(U)$, then $T^{-1}(x)\in F$ so there exists a sequence $W_{i_n}(n)$ as above that shrinks to $T^{-1}(x)$.  The fact that $T$ is conformal implies that $E(T(W_{i_n}(n)))\to 1$ and $E( f( T(W_{i_n}(n))))= E( T(f(W_{i_n}(n))))\to 1$. Summarizing $E_f(x)=1$ for all points of boundary type \ref{type:b:f}.

Altogether, we have shown that $E_f(x)=1$ for all points $x\notin G$.  By Theorem \ref{theorem:cned:removable:eccentric}, $f$ is quasiconformal on $\widehat{\C}$. The set $G$ has measure zero by \ref{c:meas}, so $E_f=1$ a.e.\ on $\widehat{\C}$.  Finally, Lemma \ref{lemma:conformal} implies that $f$ is conformal.
\end{proof}

\section{An example}\label{section:example}
Here we present an example showing that a conformal map $f$ from a circle domain $U$ onto a generalized Jordan domain $V$ whose complementary components have diameters in $\ell^2$ need not extend to a homeomorphism of the closures without any further assumption. The reason is that a circle of the boundary of $U$ might correspond to a point component of $\partial V$. Although this type of construction is known to the experts,  some further attention is required to ensure the $\ell^2$ condition. 

For simplicity, the domain $V$ in our example will be a slit domain whose boundary consists of a point and of countably many isolated non-degenerate slits that accumulate at that point. However, upon opening slightly the slits (e.g.\ quasiconformally) one can obtain a generalized Jordan domain.

\begin{figure}
\centering
\begin{tikzpicture}[scale=4]

	\draw (0,0) rectangle (1,1);
	
	\draw[fill=black!10] (0,0) -- (0.5,0)-- (0.0625,0.5)-- (0,0.5)--cycle;
	
	\draw (0,0.25)-- (0.5,0.25); \node[anchor=east] at (0,0.125) {$1/2$};
	\draw (0,0.375)--(0.5,0.375);\node[anchor=east] at (0,0.31) {$1/2^2$};
	\draw (0,0.4375)--(0.5, 0.4375);
	\draw (0,0.5)--(0.5,0.5);
	
	\node[anchor=east] (a) at (-0.1,0.7) {$1/2^3$};
	\node[anchor=east] at (-0.28,0.7) {$1/2^{n-1}=$};
	\draw[->,densely dashed] (a) to [out=-90,in=180] (-0.02,0.4);
	\draw[->,densely dashed] (a) -- (-0.02,0.47);
	\draw[->,densely dashed] (a) to [out=0, in=90] (0.03,0.51);
	
	\draw[color=blue, fill=white, ] (0.5, 0) -- (0.9375 ,0.5) -- (0.5,1) -- (0.0625,0.5)--cycle;
	\draw[<-] (1.1,0) -- (1.1,0.45) node[anchor=south] {$2$};
	\draw[->] (1.1,0.55)-- (1.1,1);

\begin{scope}[xshift=-0.2cm]
	\draw (2,0) rectangle (2.5,2);
	\draw[fill=black!10] (2,0)--(2.25,0)--(2.25,1)--(2,1)--cycle;
	\draw[blue] (2.25,0) -- (2.25, 2);
	\draw (2,1)-- (2.25,1);
	\draw (2,.5)-- (2.25,.5);
	\draw (2,.75)-- (2.25,.75);
	\draw (2,.25)-- (2.25,.25);
	\draw[<-] (2.6,0) -- (2.6,0.95) node[anchor=south] {$2$};
	\draw[->] (2.6,1.05)-- (2.6,2);
	\node[anchor=east] at (2,0.125) {$1/n=1/4$};
	\node[anchor=east] at (2,0.375) {$1/4$};
	\node[anchor=east] at (2,0.625) {$1/4$};
	\node[anchor=east] at (2,0.875) {$1/4$};
\end{scope}

\end{tikzpicture}
\caption{The sets $Q(4)$ and $S(4)$.}\label{figure:rhombus}
\end{figure}
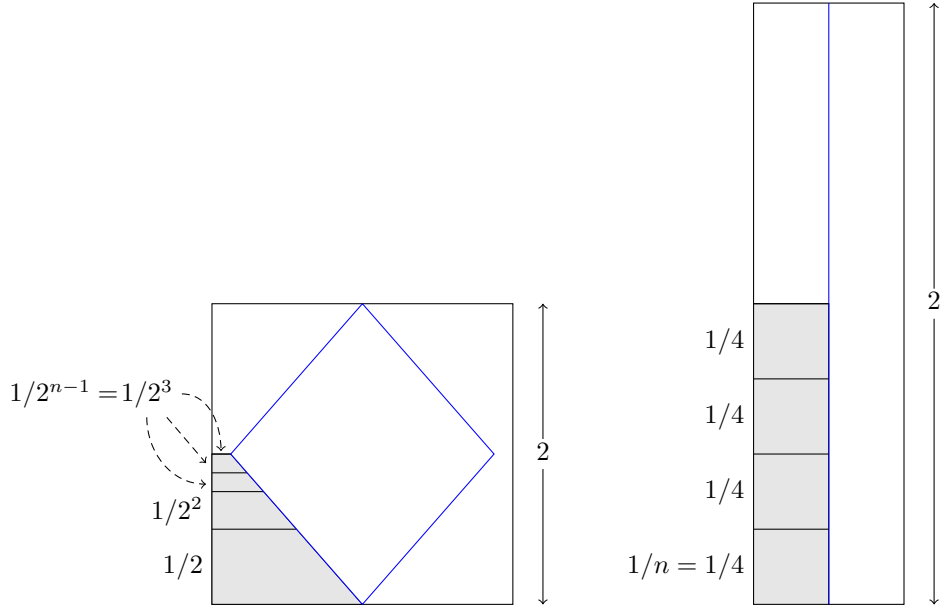

For each $n\in \N$, $n\geq 2$, consider a closed rhombus $R(n)\subset [-1,1]^2$ as in Figure \ref{figure:rhombus}, where we see $R(4)$. We set $Q(n)=[-1,1]^2\setminus R(n)$. We also consider a slit rectangle $S(n)=([-1/n,1/n]\setminus \{0\})\times [-1,1]$ as in the figure, where we see $S(4)$.  We subdivide $Q(n)$ into trapezoids, as shown in the figure. Each trapezoid in $Q(n)$ is subdivided into two triangles via a diagonal and mapped in a piecewise linear way to a corresponding square of $S(n)$. This map is uniformly quasiconformal, with distortion independent of $n$, because the angles of the triangles are uniformly bounded away from $0$. In this way we obtain a uniformly quasiconformal piecewise linear map from $Q(n)$ onto the slit rectangle $S(n)$.

\begin{figure}
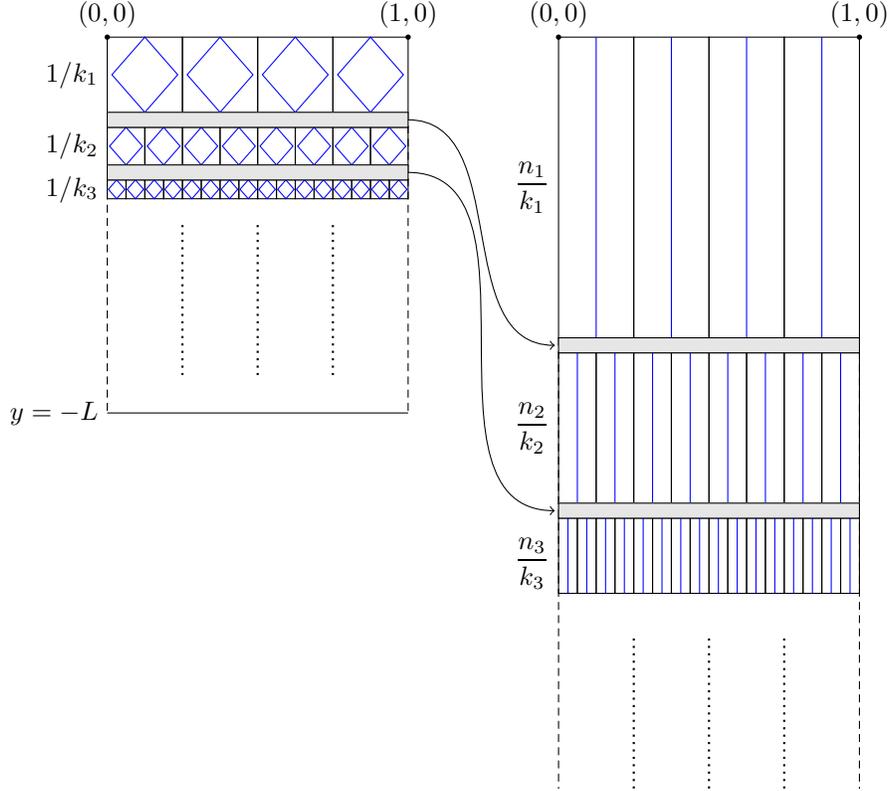

\centering
\begin{tikzpicture}[scale=1]

\input{rhombus_seed.tikz}

\begin{scope}[shift={(0,-0.7)}, scale=0.5]
\input{rhombus_seed.tikz}
\end{scope}

\begin{scope}[shift={(2,-0.7)}, scale=0.5]
\input{rhombus_seed.tikz}
\end{scope}

\begin{scope}[shift={(0,-1.15)}, scale=0.25]
\input{rhombus_seed.tikz}
\end{scope}

\begin{scope}[shift={(1,-1.15)}, scale=0.25]
\input{rhombus_seed.tikz}
\end{scope}

\begin{scope}[shift={(2,-1.15)}, scale=0.25]
\input{rhombus_seed.tikz}
\end{scope}

\begin{scope}[shift={(3,-1.15)}, scale=0.25]
\input{rhombus_seed.tikz}
\end{scope}

\draw (0,-4) node[anchor=east] {$y=-L$}--(4,-4);	
\draw[dotted, thick] (2,-1.5) --( 2, -3.5);	
\draw[dotted, thick] (1,-1.5) --( 1, -3.5);	
\draw[dotted, thick] (3,-1.5) --( 3, -3.5);	

\draw[densely dashed] (0,0) -- (0,-4);
\draw[densely dashed] (4,0)--(4,-4);

\fill (0,1) circle (1pt) node[anchor=south]  {$(0,0)$}; 
\fill (4,1) circle (1pt) node[anchor=south]  {$(1,0)$}; 

\node[anchor=east] at (0,0.5) {$1/k_1$};
\node[anchor=east] at (0,-.44) {$1/k_2$};
\node[anchor=east] at (0,-1.02) {$1/k_3$};


\begin{scope}[shift={(0, -3)}]

\begin{scope}[shift={(6,0)},scale=2]
\input{slit_seed.tikz}
\end{scope}

\begin{scope}[shift={(6,-2.2)},scale=1]
\input{slit_seed.tikz}
\end{scope}
\begin{scope}[shift={(8,-2.2)},scale=1]
\input{slit_seed.tikz}
\end{scope}

\begin{scope}[shift={(6,-3.4)},scale=.5]
\input{slit_seed.tikz}
\end{scope}
\begin{scope}[shift={(7,-3.4)},scale=.5]
\input{slit_seed.tikz}
\end{scope}
\begin{scope}[shift={(8,-3.4)},scale=.5]
\input{slit_seed.tikz}
\end{scope}
\begin{scope}[shift={(9,-3.4)},scale=.5]
\input{slit_seed.tikz}
\end{scope}

\draw[dotted, thick] (7,-4) --( 7, -6);	
\draw[dotted, thick] (8,-4) --( 8, -6);	
\draw[dotted, thick] (9,-4) --( 9, -6);	

\draw[densely dashed] (6,0) -- (6,-6);
\draw[densely dashed] (10,0)--(10,-6);

\fill (6,4) circle (1pt) node[anchor=south]  {$(0,0)$}; 
\fill (10,4) circle (1pt) node[anchor=south]  {$(1,0)$}; 

\node[anchor=east] at (6,2) {$\displaystyle{\frac{n_1}{k_1}}$};
\node[anchor=east] at (6,-1.2) {$\displaystyle{\frac{n_2}{k_2}}$};
\node[anchor=east] at (6,-3) {$\displaystyle{\frac{n_3}{k_3}}$};
\end{scope}

\draw[fill=black!10] (0,-0.2) rectangle (4,0);
\draw[fill=black!10] (0,-0.9) rectangle (4,-.7);
\draw[fill=black!10] (6,-3.2) rectangle (10,-3);
\draw[fill=black!10] (6,-5.4) rectangle (10,-5.2);

\draw[->] (4,-.1) to [out=0, in=180] ( 5.95,-3.1) ;
\draw[->] (4,-.8) to [out=0, in=180] ( 5.95,-5.3) ;

\end{tikzpicture}
\caption{The construction of the piecewise linear map $f$.}\label{figure:rhombus_map}
\end{figure}

We consider scaled and translated copies of the sets $Q(n)$ and $S(n)$, corresponding to different parameters $n$, and create a piecewise linear homeomorphism $f$ from a domain in $[0,1]\times (-L,\infty)$, namely, the complement of the rhombi, onto a domain inside $[0,1]\times (-L',\infty)$, namely, the complement of the slits, for some parameters $L,L'\in (0,\infty]$, as shown in Figure \ref{figure:rhombus_map}. Specifically, the map $f$ is set to be the identity on $[0,1]\times (0,\infty)$. The first row in the left consists of $k_1$ scaled copies of $Q(n_1)$ with height $1/k_1$ that are mapped to the first row in the right, which consists of $k_1$ scaled copies of slit rectangles $S(n_1)$. Note that the height of these rectangles is $n_1/k_1$. The map $f$ is linear on the entire bottom line of the first row. On the gray rectangle between the first and second rows we define $f$ to be a translation, mapping it to the corresponding rectangle of the same dimensions on the right side. The height of this gray rectangle, which serves as a transition zone, can be arbitrary, so we set it equal to a number $\delta_1\leq 1/k_1$. We proceed in the same way with the other rows.  Note that the map $f$ is quasiconformal in the complement of the rhombi. 

We require that 
\begin{align*}
\sum_{i=1}^\infty \frac{1}{k_i}<\infty \quad \textrm{and} \quad \sum_{i=1}^\infty \frac{n_i}{k_i}=\infty
\end{align*}
so that $L<\infty$ and $L'=\infty$. The map $f$ projects under $z\mapsto e^{-2\pi i z}$ to a quasiconformal map $F$ from a domain $U\subset \C$ outside a disk $B(0,e^{-2\pi L})$, bounded by countably many distorted rhombi  accumulating at that disk, onto a radial slit domain $V$, whose slits accumulate at the origin. 

We claim that the domain $U$ is cofat. This follows from the facts that the rhombi are uniformly fat and the map $z\mapsto e^{-2\pi i z}$ is conformal and distorts the rhombi in a controlled fashion by Koebe's distortion theorem. One can also argue directly here, by observing that for any two points $z,w$ with $|z-w|\leq 1/4$ one has
$$|e^{-2\pi i z} -e^{-2\pi i w}| \simeq e^{-2\pi i z} | z-w|$$ 
with uniform constants. This implies that any $\tau$-fat set of diameter less than $1/4$ is mapped by $z\mapsto e^{-2\pi i z}$ to a $c(\tau)$-fat set. Since $1/k_i\to 0$ as $i\to\infty$, we see that all but finitely many rhombi are small and the above estimate is true in them. On the finitely many large rhombi the map $z\mapsto e^{-2\pi i z}$ is bi-Lipschitz, so their images are also fat. 

Next, we ensure that the complementary components of $V$ have diameters in $\ell^2$. Observe that the slits at the $i$-th row have height $n_i/k_i$ and imaginary part smaller than $-\sum_{j=1}^{i-1} n_j/k_j \eqqcolon -a_{i-1}$, where $a_0=0$. Therefore, by the fundamental theorem of calculus, the image of each of them under $z\mapsto e^{-2\pi i z}$ has diameter bounded by $C e^{-2\pi a_{i-1} } n_i/k_i$, for some uniform $C>0$. The $\ell^2$ condition is implied by 
\begin{align*}
\sum_{i=1}^\infty e^{-4\pi a_{i-1}} \frac{n_i^2}{k_i^2}\cdot  k_i   <\infty. 
\end{align*}
We now choose $k_i=i^2$, $n_i=i$, so $a_i= \sum_{j=1}^i 1/j \geq \log (i+1)$. Then all above restrictions are satisfied. 

By the He--Schramm \cite{HeSchramm:Uniformization} uniformization theorem (countable Koebe's conjecture), there exists a conformal map from a circle domain onto $U$. Moreover, the cofatness of $U$ implies that all complementary components of the circle domain are non-degenerate, as shown by Schramm \cite[Theorem 4.2]{Schramm:transboundary}. Therefore, we may assume that $U$ is a circle domain and $F\colon U\to V$ is a quasiconformal map such that a circle in $\partial U$ corresponds to a point in $\partial V$.  Finally, by the work of Sibner \cite{Sibner:KoebeQC}, there exists a quasiconformal map $\phi\colon \widehat \C \to \widehat \C$ such that $\phi(U)$ is a circle domain and the map $F\circ \phi^{-1}$ is a conformal map from the circle domain $\phi(U)$ onto $V$, as desired.

\bibliography{biblio}
\end{document}